\documentclass{amsart}
\usepackage{amsmath, amssymb, mathrsfs, verbatim, multirow}
\usepackage[all]{xy}
\usepackage{pifont}
\usepackage{float}
\usepackage{color}

\newtheorem{Teo}{Theorem}[section]
\newtheorem{Prop}[Teo]{Proposition}
\newtheorem{Lema}[Teo]{Lemma}
\newtheorem{Cor}[Teo]{Corollary}

\theoremstyle{definition}
\newtheorem{Def}[Teo]{Definition}
\newtheorem{Not}[Teo]{Notation}
\newtheorem{Obs}[Teo]{Remark}

\newcommand{\N}{\mathbb{N}}

\newcommand{\lra}{\longrightarrow}

\begin{document}
\title{Limit key polynomials as $p$-polynomials}
\author{Michael de Moraes}
\author{Josnei Novacoski}
\thanks{During the realization of this project the second author was supported by a grant from Funda\c c\~ao de Amparo \`a Pesquisa do Estado de S\~ao Paulo (process number 2017/17835-9).}

\begin{abstract}
The main goal of this paper is to characterize limit key polynomials for a valuation $\nu$ on $K[x]$. We consider the set $\Psi_\alpha$ of key polynomials for $\nu$ of degree $\alpha$. We set $p$ be the exponent characteristic of $\nu$. Our first main result (Theorem \ref{maintheorem}) is that if $F$ is a limit key polynomial for $\Psi_\alpha$, then the degree of $F$ is $p^r\alpha$ for some $r\in\N$. Moreover, in Theorem \ref{them2}, we show that there exist $Q\in\Psi_\alpha$ and $F$ a limit key polynomial for $\Psi_\alpha$, such that the $Q$-expansion of $F$ only has terms which are powers of $p$.
\end{abstract}
\maketitle

\section{Introduction}
The concept of key polynomials was introduced in \cite{Mac_1} and \cite{Mac_2} in order to understand extensions of a valuation $\nu_0$ on a field $K$ to the field $K(x)$. The idea is that for a given valuation $\nu$ on $K(x)$, a key polynomial $Q\in K[x]$ for $\nu$ allows us to \textit{augment} $\nu$, i.e., build a new valuation $\nu'$ with $\nu(f)\leq \nu'(f)$ for every $f\in K[x]$ (which we denote by $\nu\leq\nu'$) and $\nu(Q)<\nu'(Q)$. MacLane proved that if $\nu_0$ is discrete and of rank one, then every valuation $\nu$ on $K(x)$, extending $\nu_0$, can be built by starting with a \textit{monomial valuation} and using a sequence (of order type at most $\omega$) built iteratively to obtain $\nu$.

A major development was presented by Vaqui\'e in \cite{Vaq_1}, \cite{Vaq_3} and \cite{Vaq_2}. He introduced the concept of limit key polynomial, and proved that if we allow these objects in the sequence, then we can drop the assumption of $\nu_0$ being discrete in MacLane's main result (in this case, the order type of the sequence can be larger than $\omega$).

An alternative definition of key polynomials was introduced in \cite{Spivmahandjul} and \cite{SopivNova} (in \cite{Spivmahandjul} they are called abstract key polynomials). The main difference between these two objects is that a key polynomial for $\nu$ (as in Maclane and Vaqui\'e's work) allows us to \textit{augment} $\nu$, while a key polynomial for $\nu$ (as in \cite{Spivmahandjul} and \cite{SopivNova}) allows us to \textit{truncate} $\nu$. In particular, if we consider the valuation $\nu'$ on $K[x]$ obtained from $\nu$ by the MacLane-Vaqui\'e's method, then $\nu\leq\nu'$ and by the method in \cite{Spivmahandjul} and \cite{SopivNova}, we obtain $\nu'\leq\nu$. Because of this, key polynomials (as in \cite{Spivmahandjul} and \cite{SopivNova}) are better to stablish the relation to other similar objects in the literature, such as \textit{pseudo-convergent sequences} as defined in \cite{Kap} and \textit{minimal pairs} as defined in \cite{Kand2}. The relations between these objects were explored in \cite{Nov} and \cite{SopivNova}.

The main purpose of the work presented here is to understand the structure of limit key polynomials. These objects are main obstacles in some problems concerning valuations. One example of this is the local uniformization problem, which is open in positive characteristic. One of the main problems in handling valuations is the existence of defect. For instance, in \cite{MoutCut} it is proved that if an extension does not have defect, then one can \textit{extend} local uniformization. Hence, one is led to ask the relation between defect and the degrees of limit key polynomials. This relation was stablished in the case of \textit{unique prolongation of the valuation} (see Corollary 6.1 of \cite{Nart} or Corollaire 2.10 of \cite{Vaq_2}). We hope to prove similar results for the case where the prolongation is not unique.

For a valuation $\nu$ on a field $K$ we denote the residue field of $\nu$ by $K\nu$. Then the \textbf{characteristic exponent of $(K,\nu)$} is defined as ${\rm char}(K\nu)$ if ${\rm char}(K\nu)>0$ and $1$ if ${\rm char}(K\nu)=0$. It is well-known that the defect of an extension is a power of $p$, the characteristic exponent of $\nu$. If one wants to relate the defect to the degrees of limit key polynomials it is natural to ask whether the degree of a limit key polynomial for $\Psi_\alpha$ is of the form $p^r\alpha$ for some $r\in \N$. This is our first main result. 

Throughout this paper we will consider a rank one valuation $\nu$ on $K[x]$. For $f\in K[x]\setminus\{0\}$ we define the \textbf{level} $\epsilon(f)$ of $f$ by
\begin{equation}\label{KP}
\epsilon(f)=\max_{b\in\mathbb{N}}\left\{{\frac{\nu(f)-\nu(\partial_bf)}{b}}\right\},
\end{equation}
where $\partial_bf$ is the Hasse derivative of $f$ of order $b$. Take $Q\in K[x]\setminus K$ a monic polynomial. We say that $Q$ is a \textbf{key polynomial} for $\nu$ if for every $f\in K[x]\setminus\{0\}$ of degree smaller than $\deg(Q)$ we have $\epsilon(f)<\epsilon(Q)$.

For any $f,q\in K[x]$, with $\deg(q)\geq 1$, since $K[x]$ is a Euclidean domain, there exist uniquely determined $f_0,\ldots,f_n\in K[x]$ with $f_i=0$ or $\deg(f_i)<\deg(q)$ for every $i$, $0\leq i\leq n$, such that
\[
f=f_nq^n+\ldots+f_1q+f_0.
\]
This expression is called the \textbf{$q$-expansion of $f$}. In this case, we set $\deg_q(f):=n$. For each $q\in K[x]$ the map
\[
\nu_q(f_nq^n+\ldots+f_1q+f_0):=\min_{0\leq i\leq n}\left\{\nu(f_iq^i)\right\}
\]
is well-defined. This map is called the \textbf{truncation of $\nu$ on $q$}. For polynomials $f,q\in K[x]\setminus\{0\}$ we define
\[
S_q(f):=\{i \mid\nu_q(f)=\nu(f_iQ^i)\}\mbox{ and }\delta_q(f):=\max S_q(f).
\]

For $\alpha\in\N$ we consider $\Psi_\alpha$ the set of key polynomials for $\nu$ of degree $\alpha$. Assume that for every $Q\in \Psi_\alpha$ there exists $Q'\in \Psi_\alpha$ such that $\epsilon(Q)< \epsilon(Q')$. Consider the set
\[
S_\alpha:=\{f\in K[x]\mid \nu_Q(f)<\nu(f)\mbox{ for every }Q\in\Psi_\alpha\}.
\]
A monic polynomial $F\in K[x]$ is a \textbf{limit key polynomial for $\Psi_\alpha$} if it belongs to $S_\alpha$ and has the least degree among polynomials in $S_\alpha$. One of the main goals of this paper is to prove the following.

\begin{Teo}\label{maintheorem}
Assume that $\Psi_\alpha$ admits a limit key polynomial. Then there exists $r\in\N$ such that for every $Q\in\Psi_\alpha$ and every limit key polynomial $F$ for $\Psi_\alpha$ the $Q$-expansion of $F$ is of the form
\[
F=Q^{p^r}+a_{p^r-1}Q^{p^{r}-1}+\ldots+a_1Q+a_0.
\]
Moreover, for every $Q\in\Psi_\alpha$, we have $\nu_Q(F)=p^r\nu(Q)$.
\end{Teo}

The idea of the proof of Theorem \ref{maintheorem} is the following. We let
\[
F=a_d Q^d+a_{d -1}Q^{d-1}+\ldots+a_0
\]
be the $Q$-expansion of $F$. We set $\delta:=\delta_Q(F)$. Then we show that $d=\delta$ (Proposition \ref{delta=deg}). This means that
\[
F=a_\delta Q^\delta+a_{\delta -1}Q^{\delta-1}+\ldots+a_0.
\]
Then we show that $a_\delta=1$ (Proposition \ref{coeficientelider1}) and that $\delta=p^r$ for some $r\in\N$ (Proposition \ref{Lambda} \textbf{(vi)}).

The relation of key polynomials and \textit{pseudo-convergent sequences} was studied in \cite{SopivNova}. There, it is shown that any polynomial of smallest degree \textit{not fixed} by a pseudo-convergent sequence is a limit key polynomial. Moreover, in \cite{Kap} it is shown that such polynomials can be chosen to be \textit{$p$-polynomials}, i.e., of the form
\[
a_{p^r}x^{p^r}+a_{p^{r-1}}x^{p^{r-1}}\ldots+a_1x+a_0.
\]
Our next result generalizes this for limit key polynomial of any degree.

\begin{Teo}\label{them2}
Assume that $\Psi_\alpha$ admits a limit key polynomial. Then there exist $r\in\N$, $Q\in\Psi_\alpha$ and a limit key polynomial $F$ for $\Psi_\alpha$ such that the $Q$-expansion of $F$ is of the form
\[
F=Q^{p^r}+a_{p^{r-1}}Q^{p^{r-1}}+\ldots+a_1Q+a_0.
\]
\end{Teo}

The main idea for the proof of Theorem \ref{them2} is the following. Since ${\rm rk}(\nu)=1$ and the set $\nu(\Psi_\alpha)$ is bounded it has a supremum. We set $B=\sup(\nu(\Psi_\alpha))$. Then, for $Q\in\Psi_\alpha$ and $F$ a limit key polynomial for $\Psi_\alpha$, if a monomial $a_iQ^i$ in the $Q$-expansion of $F$ has value greater than $p^rB$, then $F-a_iQ^i$ is also a limit key polynomial for $\Psi_\alpha$ (Lemma \ref{4removedordemonomios}). We then study the behaviour of the values of the coefficients in the expansions of $F$ in elements of $\Psi_\alpha$. Using this we show that, for sufficiently large $Q\in\Psi_\alpha$ the terms $a_iQ_i$ in the $Q$-expansion of $F$ for which $i$ is not a power of $p$ (and $i\neq 0$) have value greater than $p^rB$. Hence, we can eliminate them to obtain the desired limit key polynomial for $\Psi_\alpha$. Theorems \ref{maintheorem} and \ref{them2} will follow as immediate consequence of Proposition \ref{Lambda}.

The reason to present a form of limit key polynomials as in Theorem \ref{them2} is because the roots of such polynomials are simpler. Another reason for this is to classify the defect of an extension. In \cite{Kuhl}, Kuhlmann presents a classification of Artin-Schreier defect extensions as dependent or independent. We hope that the characterization of limit key polynomials as in Theorem \ref{them2} will allow us to present similar classifications for defect extensions which are not necessarily Artin-Schreier.

We observe that the results presented here are not new. For instance, Theorems \ref{maintheorem} and \ref{them2} appear in \cite{MaSpiv}. The main difference is that there they consider a specific type of sequence of key polynomials while here we consider the set of all key polynomials. We point out that \cite{MaSpiv} was a source of ideas for this paper. However, some steps in \cite{MaSpiv} are not clear to us, so we developed alternative proofs. We believe that our version of these proofs are simpler and clearer. Also, Theorem \ref{maintheorem} is Theorem 3.5 of \cite{Vaq_3} and a generalization of it to valuations of arbitrary rank can be found in \cite{AFFGNR} (Theorem 4.9).
\\
\par\medskip
\textbf{Acknowledgements.} We would like to thank the anonymous referee for a careful reading, for providing useful suggestions and for pointing out a few mistakes in an earlier version of this paper.

\section{Preliminaires}
\begin{Def}
Take a commutative ring $R$ with unity. A \index{Valuation}\textbf{valuation} on $R$ is a mapping $\nu:R\lra \Gamma_\infty :=\Gamma \cup\{\infty\}$ where $\Gamma$ is an ordered abelian group (and the extension of addition and order to $\infty$ in the obvious way), with the following properties:
\begin{description}
\item[(V1)] $\nu(fg)=\nu(f)+\nu(g)$ for every $f,g\in R$.
\item[(V2)] $\nu(f+g)\geq \min\{\nu(f),\nu(g)\}$ for every $f,g\in R$.
\item[(V3)] $\nu(1)=0$ and $\nu(0)=\infty$.
\end{description}
\end{Def}

\subsection{Key polynomials}

In this section we discuss the basics about key polynomials as presented in \cite{SopivNova}.

\begin{Obs}
The equality \eqref{KP} implies that
\begin{equation}\label{UT}
\nu(\partial_bf)\geq\nu(f)-b\epsilon(f),
\end{equation}
for every $b\in\mathbb{N}$. Hence, if $\epsilon(g)>\epsilon(f)$, then for every $b\in\mathbb{N}$ we have
\[
\nu(\partial_bf)>\nu(f)-b\epsilon(g).
\]
\end{Obs}

\begin{Not}
For $f\in K[x]\setminus\{0\}$ we denote by
\[
I(f):=\{b\in\mathbb{N}\mid \mbox{the equality holds in } \eqref{UT}\}.
\]
\end{Not}

\begin{Lema}[Corollary 4.4 of \cite{Matheus}]\label{paraprovarprod}
Let $f,g\in K[x]\setminus\{0\}$. We have
\begin{equation}\label{epsilonmultil}
\epsilon(fg)=\max\{\epsilon(f),\epsilon(g)\}.
\end{equation}
\end{Lema}

\begin{Cor}\label{irredutivel}
If $Q\in K[x]$ is a key polynomial for $\nu$, then $Q$ is irreducible. 
\end{Cor}

\begin{proof}
Suppose, aiming for a contradiction, that $Q$ is not irreducible. Write $Q=fg$ with $f,g\in K[x]$ of degree smaller that $\deg(Q)$. By Lemma \ref{paraprovarprod} we have $\epsilon(Q)=\epsilon(f)$ or $\epsilon(Q)=\epsilon(g)$. This is a contradiction to the definition of key polynomial.
\end{proof}

For the remaining of this section, let $Q$ be a key polynomial for $\nu$ and set $\epsilon:=\epsilon(Q)$.

\begin{Obs}\label{1card(SQ(f))=1}
The following properties are satisfied.
\begin{description}
\item[(i)] If $S_{Q}(f)$ is a singleton, then $\nu_{Q}(f)=\nu(f)$.
\item[(ii)] If $g\in K[x]$ is such $\nu_{Q}(g)<\nu_{Q}(f)$, then
\[
\nu_{Q}(f+g)=\nu_{Q}(g)\mbox{ and }S_Q(f+g)=S_Q(g).
\]
\item[(iii)] If $\nu_Q(g)\leq\nu_Q(f)$ and $\delta_{Q}(g)>\delta_Q(f)$, then
\[
\nu_Q(f+g)=\nu_Q(g)\mbox{ and }\delta_Q(g)=\delta_Q(f+g).
\] 
\end{description}
\end{Obs}

\begin{Lema}\label{prod}
Let $f\in K[x]$ be a polynomial such that $\epsilon(f)<\epsilon(Q)$ and $f=qQ+r$ the $Q$-expansion of $f$. We have
\[
\nu(f)=\nu(r)<\nu(qQ),
\]
e hence $\delta_Q(f)=0$.
\end{Lema}

\begin{proof}
Let $\gamma=\min\{\nu(f),\nu(r)\}$. It is enough to show that $\nu(qQ)>\gamma$. Since the degrees of $\epsilon(f)<\epsilon(Q)$, $\deg(r)M\deg(Q)$ and $Q$ is a key polynomial, Lemma \ref{paraprovarprod} gives us
\[
\epsilon=\epsilon(qQ)>\epsilon':=\max\{\epsilon(f),\epsilon(r)\}.
\]
Let $b\in I(qQ)$. Then
\[
\nu(qQ)-b\epsilon=\nu(\partial_b(qQ))\geq\min\{\nu(\partial_b(f)),\nu(\partial_br)\}\geq\gamma-b\epsilon'>\gamma-b\epsilon,
\]
and hence $\nu(qQ)>\gamma$.
\end{proof}

\begin{Obs}\label{obsdesnecessaria}
Let $f,g\in K[x]\setminus\{0\}$ be polynomials of degree smaller than $\deg(Q)$ and $n,m\in\mathbb{N}_0$. Then
\begin{equation}\label{eqsobremichelnovo}
\nu_Q(fQ^ngQ^m)=\nu(fQ^n)+\nu(gQ^m)\mbox{ and }\delta_Q(fQ^ngQ^m)=n+m.
\end{equation}
Indeed, let $fg=qQ+r$ be the $Q$-expansion of $fg$. By Lemma \ref{prod} we have $\nu(qQ)>\nu(r)=\nu(fg)$. Since $fQ^ngQ^m=qQ^{n+m+1}+rQ^{n+m}$, we have
\[
\nu\left(qQ^{n+m+1}\right)>\nu\left(rQ^{n+m}\right)=\nu(fQ^n)+\nu(gQ^m).
\]
Hence, \eqref{eqsobremichelnovo} follows.
\end{Obs}

\begin{Lema}\label{deltaprod}
Let $f,g\in K[x]\setminus\{0\}$. Then
\[
\delta_Q(fg)=\delta_Q(f)+\delta_Q(g)\mbox{ and }\nu_Q(fg)=\nu_Q(f)+\nu_Q(g).
\]
Moreover, if $S_Q(f)$ and $S_Q(g)$ are singletons, then $S_Q(f+g)=\{\delta_Q(f)+\delta_Q(g)\}$.
\end{Lema}

\begin{proof}
Let
\[
f=f_nQ^n+\ldots+f_0\mbox{ and }g=g^mQ^m+\ldots+g_0
\]
be the $Q$-expansions of $f$ and $g$, respectively. By Remark \ref{obsdesnecessaria} and Remark \ref{1card(SQ(f))=1} \textbf{(ii)} and \textbf{(iii)} applied iteractively to the sum
\[
\sum_{i=0}^n\sum_{j=0}^m f_iQ^ig_jQ^j
\]
we obtain the result.
\end{proof}

\begin{Prop}[Proposition 2.6 of \cite{SopivNova}]
The map $\nu_Q$ is a valuation of $K[x]$.
\end{Prop}

\subsection{Key polynomials of the same degree}

Let $Q_1,Q_2\in K[x]$ be key polynomials for $\nu$ of the same degree. Assume that $\nu(Q_2)\geq\nu(Q_1)$ and let $Q_2=Q_1+h$ be the $Q_1$-expansion of $Q_2$. Since $\nu(Q_2)\geq\nu(Q_1)$, we have
\[
\nu(h)\geq\nu_{Q_1}(Q_2)=\nu_{Q_2}(Q_1)=\nu(Q_1)
\]
and hence $\delta_{Q_1}(Q_2)=1$. Since $\deg(h)<\deg(Q)$, we have $\epsilon(h)<\epsilon(Q)$.

\begin{Lema}\label{2comparacao}
We have the following.
\begin{description}
\item[(i)] If $b\in I(Q_1)$, then $\nu(\partial_bQ_2)=\nu(\partial_bQ_1)$.
\item[(ii)] If $\nu(Q_1)=\nu(Q_2)$, then $\epsilon(Q_1)=\epsilon(Q_2)$ and $I(Q_1)=I(Q_2)$.
\item[(iii)] If $\nu(Q_1)<\nu(Q_2)$, then $\epsilon(Q_1)<\epsilon(Q_2)$.
\item[(iv)] Let $b_1\in I(Q_1)$ and $b_2\in I(Q_2)$. If
\[
\nu(Q_1)<\nu(Q_2),\nu(\partial_{b_2}Q_2)=\nu(\partial_{b_2}Q_1)\mbox{ and }\nu(\partial_{b_1}Q_2)=\nu(\partial_{b_1}Q_1),
\]
then $b_2\leq b_1$.
\end{description}
\end{Lema}

\begin{proof}
Since
\[
\nu(\partial_bh)\geq\nu(h)-b\epsilon(h)>\nu(h)-b\epsilon(Q_1)\geq\nu(Q_1)-b\epsilon(Q_1)=\nu(\partial_bQ_1)
\]
we have
\[
\nu(\partial_bQ_2)=\nu(\partial_bQ_1+\partial_b h)=\nu(\partial_bQ_1).
\]
This proves \textbf{(i)}.

In order to prove \textbf{(ii)}, let $b\in I(Q_1)$ or $b\in I(Q_2)$. Since $\nu(Q_2)=\nu(Q_1)$, by \textbf{(i)} we have
\[
\nu(\partial_bQ_2)=\nu(\partial_bQ_1)=\nu(Q_1)-b\epsilon(Q_1)=\nu(Q_2)-b\epsilon(Q_1).
\]
Hence $\epsilon(Q_2)=\epsilon(Q_1)$ and $I(Q_2)=I(Q_1)$.

Take $b\in I(Q_1)$. By \textbf{(i)} we have $\nu(\partial_bQ_2)=\nu(\partial_bQ_1)$. Since $\nu(Q_2)>\nu(Q_1)$ we have
\[
\nu(\partial_bQ_2)=\nu(\partial_bQ_1)=\nu(Q_1)-b\epsilon(Q_1)<\nu(Q_2)-b\epsilon(Q_1).
\]
Hence $\epsilon(Q_2)>\epsilon(Q_1)$ and this proves \textbf{(iii)}.

For \textbf{(iv)}, suppose aiming for a contradiction, that $b_2>b_1$. Since $\nu(\partial_{b_2}Q_2)=\nu(\partial_{b_2}Q_1)$, $\nu(\partial_{b_1}Q_2)=\nu(\partial_{b_1}Q_1)$ and
\[
\epsilon(Q_1)=\frac{\nu(Q_1)-\nu(\partial_{b_1}Q_1)}{b_1}\geq\frac{\nu(Q_1)-\nu(\partial_{b_2}Q_1)}{b_2},
\]
we have
\[
\frac{\nu(Q_1)-\nu(Q_2)}{b_1}+\frac{\nu(Q_2)-\nu(\partial_{b_1}Q_2)}{b_1}\geq\frac{\nu(Q_1)-\nu(Q_2)}{b_2}+\frac{\nu(Q_2)-\nu(\partial_{b_2}Q_2)}{b_2}.
\]
Since $\nu(Q_2)>\nu(Q_1)$ and $b_2>b_1$, we have
\[
\frac{\nu(Q_1)-\nu(Q_2)}{b_1}<\frac{\nu(Q_1)-\nu(Q_2)}{b_2}.
\]
Therefore
\[
\frac{\nu(Q_2)-\nu(\partial_{b_1}Q_2)}{b_1}>\frac{\nu(Q_2)-\nu(\partial_{b_2}Q_2)}{b_2},
\]
and this is a contradiction to $b_2\in I(Q_2)$.
\end{proof}

\begin{Lema}\label{2sujestao}
Let $f\in K[x]\setminus\{0\}$, $f=f_nQ_2^n+\ldots+f_0$ the $Q_2$-expansion of $f$ and $r$, $0\leq r\leq n$, the largest natural number such that
\[
\nu_{Q_1}(f_rQ_2^r)\leq\nu_{Q_1}(f_iQ_2^i)\mbox{ for every }i,0\leq i\leq n.
\]
Then we have the following.
\begin{description}
\item[(i)] $\nu_{Q_1}(f)\leq\nu_{Q_2}(f)\leq\nu(f)$.
\item[(ii)]  $\nu_{Q_1}(f)=\nu_{Q_1}(f_rQ_2^r)=\nu(f_rQ_1^r)$ and $\delta_{Q_1}(f)=r$.
\item[(iii)] The coefficients of degree $\delta_{Q_1}(f)$ in the $Q_1$ and $Q_2$-expansions of $f$ have the same value.
\item[(iv)]  $\delta_{Q_2}(f)\leq\delta_{Q_1}(f)$.
\item[(v)]  The leading coefficients of the $Q_1$ and $Q_2$-expansions of $f$ have the same value.
\item[(vi)] If $\delta_{Q_1}(f)=0$, then
\[
\nu(f_iQ_2^i)-\nu(f_0)>i(\nu(Q_2)-\nu(Q_1))\geq\nu(Q_2)-\nu(Q_1),
\]
for every $i$, $1\leq i\leq n$.
\item[(vii)]  If $\delta_{Q_2}(f)>0$ and $\nu(Q_1)<\nu(Q_2)$, then $\nu_{Q_1}(f)<\nu_{Q_2}(f)$.
\end{description}
\end{Lema}

\begin{proof}
For \textbf{(i)}, it is enough to write the $Q_1$-expansion of $f$ and use the fact that $\nu_{Q_2}(Q_1)=\nu(Q_1)$. 

Since $\delta_{Q_1}(Q_2)=1$, by Lemma \ref{deltaprod} we have that
\[
\delta_{Q_1}(f_iQ_2^i)=i\mbox{ for every }i, 0\leq i\leq n.
\]
Hence, Remark \ref{1card(SQ(f))=1} \textbf{(ii)} and \textbf{(iii)} applied (iteractively) to the sum $f_0+\ldots+f_nQ_2^n$ give us \textbf{(ii)}.

The item \textbf{(iii)} follows immediately from \textbf{(ii)}.

For every $i$, $r+1\leq i\leq n$, since $\nu_{Q_1}(f_rQ_2^r)<\nu_{Q_1}(f_iQ_2^i)$, we have $\nu(f_rQ_2^r)<\nu(f_iQ_2^i)$. Hence, by \textbf{(ii)} we have $\delta_{Q_2}(f)\leq r=\delta_{Q_1}(f)$. This shows \textbf{(iv)}.

For every $i$, $0\leq i\leq n-1$, the $Q_1$-expansion of $f_iQ_2^i$ has no term of degree $n$. Hence, in order to prove \textbf{(v)} it is enough to study the term of degree $n$ in the $Q_1$-expansion of $f_nQ_2^n$. By \textbf{(iv)}, since $\deg_{Q_1}(f_nQ_2^n)=n$, we have
\[
n=\delta_{Q_2}(f_nQ_2^n)=\delta_{Q_1}(f_nQ_2^n).
\]
Hence, \textbf{(v)} follows from \textbf{(iii)}.

Since $\delta_{Q_1}(f)=0$, \textbf{(ii)} gives us that $r=0$ and $\nu_{Q_1}(f_iQ_2^i)>\nu(f_0)$ for every $i$, $1\leq i\leq n$. Since
\[
\nu_{Q_1}\left(f_iQ_2^i\right)=\nu\left(f_iQ_2^i\right)-i(\nu(Q_2)-\nu(Q_1)),
\]
we have \textbf{(vi)}.

By \textbf{(ii)} and \textbf{(iv)} we have $0<\delta_{Q_2}(f)\leq\delta_{Q_1}(f)=r$. Since $\nu(Q_1)<\nu(Q_2)$ and $r>0$, the item \textbf{(vii)} follows from \textbf{(ii)}.
\end{proof}

\subsection{The relation between the derivatives and the $Q$-truncation}

For this section, let $Q\in K[x]$ be a key polynomial for $\nu$, $\epsilon=\epsilon(Q)$, $h\in K[x]\setminus\{0\}$ a polynomial of degree smaller than $\deg(Q)$ and $n\in\mathbb{N}$. We will study the relation between the partial derivatives $\partial_b$ and the $Q$-truncation of $\nu$.

For $b\in\mathbb{N}$, consider the set $\mathcal{S}_{b,n}$ of all tuples of the form $\gamma=(b_0,\ldots,b_r)$ where
\[
0\leq r\leq n, 0\leq b_0, 0<b_1\leq b_2\leq\ldots\leq b_r\mbox{ and }b_0+\ldots+b_r=b.
\]
We denote
\[
C(\gamma):=\frac{n!}{(n-r)!n_1!\cdots n_k!},
\]
if $\{b_1,\ldots,b_r\}$ has $k$ distinct elements $b_{i_1},\ldots, b_{i_k}$ and for each $j$, $1\leq j\leq k$, $n_j$ is the number of $b_i$'s which are equal to $b_{i_j}$. Observe that if $p\mid{n\choose r}$, then $p\mid C(\gamma)$.

By the Leibniz rule for derivation we have
\begin{equation}\label{2derivada}
\partial_b(hQ^n)=\sum_{\gamma\in \mathcal S_{b,n}}C(\gamma)T_{\gamma}(hQ^n),\mbox{ where }T_{\gamma}(hQ^n)=\partial_{b_0}(h)\left(\prod_{j=1}^r\partial_{b_i}(Q)\right)Q^{n-r}.
\end{equation}
For simplicity, we will denote $T_\gamma=T_\gamma(hQ^n)$. By Lemma \ref{deltaprod} we have $S_Q(T_\gamma)=\{n-r\}$. Consequently, by Remark \ref{1card(SQ(f))=1} \textbf{(i)} we have $\nu_Q(T_\gamma)=\nu(T_\gamma)$.

\begin{Lema}\label{2T1}
Let $b\in\mathbb{N}$ and $\gamma=(b_0,\ldots,b_r)\in\mathcal{S}_{b,n}$. We have
\begin{equation}\label{inequaliapeare}
\nu(T_{\gamma})\geq\nu(hQ^n)-b\epsilon.
\end{equation}
Moreover, the equality holds in \eqref{inequaliapeare} if and only if $b_0=0$ and $b_i\in I(Q)$ for every $i$, $1\leq i\leq r$.
\end{Lema}

\begin{proof}
For every $i$, $1\leq i\leq r$, we have
\begin{equation}\label{2usaraqui2}
\nu(\partial_{b_0}h)\geq\nu(h)-b_0\epsilon
\mbox{ and }
\nu(\partial_{b_i}Q)\geq\nu(Q)-b_i\epsilon.
\end{equation}
Since $b_0+\ldots+b_r=b$, adding \eqref{2usaraqui2}, we have
\[
\nu(T_{\gamma})=\nu\left(\partial_{b_0}(h)\left(\prod_{j=1}^r\partial_{b_i}(Q)\right)Q^{n-r}\right)\geq\nu(hQ^n)-b\epsilon.
\]
Since all the inequalities in \eqref{2usaraqui2} are equalities if and only if $b_0=0$ and $b_i\in I(Q)$ for every $i$, $1\leq i\leq r$, the result follows.
\end{proof}

\begin{Cor}\label{2vq>=v-be}
For every $b\in\mathbb{N}_0$ we have:
\begin{description}
\item[(i)] $\nu_Q(\partial_b(hQ^n))\geq\nu(hQ^n)-b\epsilon$; and
\item[(ii)] if $f\in K[x]$, then $\nu_{Q}(\partial_bf)\geq\nu_Q(f)-b\epsilon$.
\end{description}
\end{Cor}

\begin{proof}
Since
\[
\partial_b(hQ^n)=\sum_{\gamma\in \mathcal S_{b,n}}C(\gamma)T_{\gamma},
\]
Lemma \ref{2T1} gives us that $\nu_Q(\partial_b(hQ^n))\geq\nu(hQ^n)-b\epsilon$ and \textbf{(i)} follows.

In order to prove \textbf{(ii)}, let $f=f_mQ^m+\ldots+f_0$ be the $Q$-expansion of $f$. Since $\deg(f_0)<\deg(Q)$ and $Q$ is a key polynomial for $\nu$ we have $\nu_Q(\partial_bf_0)\geq\nu(f_0)-b\epsilon$. By \textbf{(i)}, we have that
\[
\nu_Q(\partial_b(f_iQ^i))\geq\nu(f_iQ^i)-b\epsilon\mbox{ for every }i,1\leq i\leq m.
\]
Since
\[
\partial_b(f)=\partial_b(f_mQ^m)+\ldots+\partial_b(f_0)
\]
we obtain \textbf{(ii)}.
\end{proof}

\begin{Lema}\label{5importante}
Let $b_M=\max I(Q)$ and $b\in\mathbb{N}$ such that $b_M|b$. If
\[
\nu_Q(\partial_b(hQ^n))=\nu(hQ^n)-b\epsilon,
\]
then
\begin{equation}\label{equalimalucaqun}
\delta_Q(\partial_b(hQ^n))\leq n-b/b_M.
\end{equation}
Moreover, the equality holds in \eqref{equalimalucaqun} if and only if $\displaystyle p\nmid{n\choose b/b_M}$.
\end{Lema}

\begin{proof}
By Lemma \ref{2T1}, the tuple $\lambda\in\mathcal{S}_{b,n}$ with the least number of coordinates such that $\nu(T_\lambda)=\nu(hQ^n)-b\epsilon$ is $\lambda=(0,b_M,\ldots,b_M)$ with $b_M$ appearing $b/b_M$ times. Since, by Lemma \ref{2T1}, $\nu(T_\gamma)\geq\nu(hQ^n)-b\epsilon$ for every $\gamma\in\mathcal{S}_{b,n}$, we have
\[
\delta_Q(\partial_b(hQ^n))\leq\{n-b/b_M\}.
\]
Since $\displaystyle C(\lambda)={n\choose b/b_M}$ and
\[
\partial_b(hQ^n)=\sum_{\gamma\in \mathcal S_{b,n}\setminus\{\lambda\}}C(\gamma)T_{\gamma}+C(\lambda)T_\lambda,
\]
we have $\delta_Q(\partial_b(hQ^n))=n-b/b_M$ if and only if $\displaystyle p\nmid{n\choose b/b_M}$. Therefore, the result follows.
\end{proof}

\begin{Prop}\label{2dp}
Let $b_M=\max I(Q)$ and $f\in K[x]\setminus\{0\}$ such that $\delta:=\delta_Q(f)>0$. Write $\delta=p^eu$, where $e\in\mathbb{N}_0$, $u\in\mathbb{N}$ and $p\nmid u$. Set $b=p^eb_M$. We have
\[
\nu_Q(\partial_bf)=\nu_Q(f)-b\epsilon\mbox{ and }\delta_Q(\partial_bf)=\delta-p^e.
\]
\end{Prop}

\begin{proof}
Let $f=f_0+\ldots+f_mQ^m$ be the $Q$-expansion of $f$. We have
\[
\partial_b(f)=\partial_bf_0+\ldots+\partial_b(f_mQ^m).
\]
Observe that $p^e=b/b_Q$. Fix $i$, $0\leq i\leq m$.

If $i=\delta$, then by Lemma \ref{5importante} and the fact that $\displaystyle p\nmid{\delta\choose p^e}$ we obtain
\[
\nu(\partial_b(f_\delta Q^\delta))=\nu(f_\delta Q^\delta)-b\epsilon\mbox{ and }\delta_Q(\partial_b(f_\delta Q^\delta))=\delta-p^e.
\]

If $i=0$, then by the fact that $\deg(f_0)<\deg(Q)$, $Q$ is a key polynomial for $\nu$ and $b>0$, we have
\[
\nu_Q(\partial_bf_0)>\nu(f_0)-b\epsilon\geq\nu(f_\delta Q^\delta)-b\epsilon=\nu(\partial_b(f_\delta Q^\delta)).
\]
The second inequality holds because $\delta=\delta_Q(f)$.

If $0<i<\delta$, then we have
\[
\nu_Q(\partial_b(f_iQ^i))\geq\nu(f_iQ^i)-b\epsilon\geq\nu(f_\delta Q^\delta)-b\epsilon=\nu(\partial_b(f_\delta Q^\delta)).
\]
By Lemma \ref{5importante}, if the first inequality is an equality, then
\[
\delta_Q(\partial_b(f_iQ^i))\leq i-p^e<\delta-p^e.
\]

For $i>\delta$, we have
\[
\nu_Q(\partial_b(f_iQ^i))\geq\nu(f_iQ^i)-b\epsilon>\nu(f_\delta Q^\delta)-b\epsilon=\nu(\partial_b(f_\delta Q^\delta))
\]
where the second equality holds because $\delta=\delta_Q(f)$.

Using the cases $0<i<\delta$ and $i=\delta$ we obtain, using Remark \ref{1card(SQ(f))=1} \textbf{(iii)}, that
\[
\nu_Q(\partial_b(f_1Q+\ldots+f_\delta Q^\delta))=\nu_Q(\partial_b(f_\delta Q^\delta))\mbox{ and }\delta_Q(\partial_b(f_\delta Q^\delta))=\delta-p^e.
\]
Hence, by the cases $i=0$ and $i>\delta$, by Remark \ref{1card(SQ(f))=1} \textbf{(ii)}, we obtain that
\[
\nu_Q(\partial_bf)=\nu_Q(\partial_b(f_\delta Q^\delta))=\nu_Q(f)-b\epsilon\mbox{ and }\delta_Q(\partial_b f)=\delta-p^e.
\]
The result now follows.
\end{proof}

\subsection{The set of key polynomials of the same degree}

Let $\Psi_\alpha$ be the set of the key polynomials for $\nu$ of degree $\alpha$. We say that $\Psi\subseteq\Psi_\alpha$ is a \textbf{final subset} of $\Psi_\alpha$ if there exists $Q\in\Psi$ such that, for every $Q'\in\Psi_\alpha$ with $\nu(Q')\geq\nu(Q)$, we have $Q'\in\Psi$.

For this section, we will assume that $\Psi_\alpha$ is non-empty and $\nu(\Psi_\alpha)$ does not have a maximum and is bounded. Since $\nu$ is a rank one valuation, we can set
\[
B:=\sup\nu(\Psi_\alpha)\mbox{ and }\epsilon(B):=\sup\epsilon(\Psi_\alpha).
\]

\begin{Prop}\label{bfinal}
There exists $b_\infty\in\mathbb{N}$ and a final subset $\Psi\subseteq\Psi_\alpha$, such that for every $Q\in\Psi$ we have $I(Q)=\{b_\infty\}$.
\end{Prop}

\begin{proof}
Since $\nu(\Psi_\alpha)$ does not have a maximum,  the result follows from Lemma \ref{2comparacao} \textbf{(i)} and \textbf{(iv)}.
\end{proof}

For the remaining of this section, let $b_\infty$ as in Proposition \ref{bfinal}, $Q_0,Q_1\in\Psi_\alpha$ such that $\epsilon(Q_1)-\epsilon(Q_0)>B-\nu(Q_1)$ and $\Psi\subset\Psi_\alpha$ a final subset of $\Psi_\alpha$ of all $Q\in\Psi_\alpha$ such that $\nu(Q)\geq\nu(Q_1)$ and $I(Q)=\{b_\infty\}$.

If $b\in\mathbb{N}$ and $h\in K[x]$ is such that $\deg(h)<\alpha$ and $\nu(h)\geq\nu(Q_1)$, then
\[
\nu(\partial_bh)\geq\nu(h)-b\epsilon(h)>\nu(Q_1)-b\epsilon(Q_0)>B-b\epsilon(Q_1)>B-b\epsilon(B).
\]

\begin{Lema}\label{B-bepsilonB}
For every $Q\in\Psi$ and $b\in\mathbb{N}$, we have $\nu(\partial_bQ)\geq B-b\epsilon(B)$.
\end{Lema}

\begin{proof}
Let $b\in\mathbb{N}$. For each $Q\in\Psi$, let $Q=Q_1+h_Q$ be the $Q_1$-expansion of $Q$. Since $\nu(h_Q)\geq\nu(Q_1)$, we have $\nu(h_Q)>B-b\epsilon(B)$. If $\nu(\partial_bQ_1)>B-b\epsilon(B)$, then the result follows from the fact that $\partial_bQ=\partial_bQ_1+\partial_bh_Q$. If $\nu(\partial_bQ_1)\leq B-b\epsilon(B)$, then we have
\[
\nu(\partial_bQ)=\min\left\{\nu(\partial_bQ_1),\nu(\partial_bh_Q)\right\}=\nu(\partial_bQ_1),\mbox{ for every }Q\in\Psi.
\]
Since $\nu(\partial_bQ_1)=\nu(\partial_bQ)\geq\nu(Q)-b\epsilon(Q)$, taking the limit in $Q$, we have
\[
\nu(\partial_bQ)=\nu(\partial_bQ_1)\geq B-b\epsilon(B)\mbox{ for every }Q\in\Psi.
\]
The result follows.
\end{proof}

\begin{Lema}\label{Lemafinal}
Let $l,n,m\in\mathbb{N}_0$ with $0<m<n$ and $a\in K[x]$ with $\deg(a)<\alpha$ such that
\[
\nu(a)\geq lB, Q\in\Psi, b=mb_\infty, \gamma=(b_0,\ldots,b_r)\in\mathcal{S}_{b,n}\mbox{ and }T_\gamma=T_\gamma(aQ^n).
\]
Then
\[
\nu(T_\gamma)\geq(l+r-m)B+(m+n-r)\nu(Q)-b\epsilon(Q).
\]
\end{Lema}

\begin{proof}
Since $\nu(a)\geq lB$ we have
\[
\nu(\partial_{b_0}a)\geq lB-b_0(\epsilon(B)-\epsilon(Q))-b_0\epsilon(Q).
\]
By Lemma \ref{B-bepsilonB} we have
\[
\nu(\partial_{b_i}Q)\geq B-b_i(\epsilon(B)-\epsilon(Q))-b_i\epsilon(Q), \mbox{ for every }i, 1\leq i\leq r.
\]
Adding these inequalities, taking into account that
\[
-b(\epsilon(B)-\epsilon(Q))=m(-B+\nu(Q)),
\]
we obtain the result.
\end{proof}

\section{Successor key polynomial}

For this section, let $\Psi_\alpha$ be the set of all key polynomials for $\nu$ of degree $\alpha$. Assume that $\Psi_\alpha$ is not empty. We say that $Q\in\Psi_\alpha$ is a \textbf{maximum element} of $\Psi_\alpha$ if for every $Q'\in\Psi_\alpha$ we have $\nu(Q')\leq\nu(Q)$.

\begin{Prop}\label{4naofixados}
If $F\in K[x]$ is monic of the smallest degree among polynomials in $S_\alpha$,
then $F$ is a key polynomial for $\nu$ of degree greater than $\alpha$. Moreover, if $g\in K[x]\setminus\{0\}$ has degree smaller than $\deg(F)$, then there exists $Q\in\Psi_\alpha$ such that $\epsilon(g)\leq\epsilon(Q)$.
\end{Prop}

\begin{proof}
By assumption and Lemma \ref{2sujestao} \textbf{(i)}, there exists a final subset $\Psi$ of $\Psi_\alpha$ such that
\[
\nu_Q(\partial_bF)=\nu(\partial_bF),\mbox{ for every }b\in\mathbb{N}\mbox{ and }Q\in\Psi.
\]
For every $Q\in\Psi_\alpha$, since $\nu_Q(F)<\nu(F)$, by Remark \ref{1card(SQ(f))=1} \textbf{(i)}, we have $S_Q(F)\neq\{0\}$. Hence, by Proposition \ref{2dp}, for every $Q\in\Psi$ there exists $b\in\mathbb{N}$ such that
\[
\nu(\partial_bF)=\nu_Q(\partial_bF)=\nu_Q(F)-b\epsilon(Q)<\nu(F)-b\epsilon(Q).
\]
This implies that $\epsilon(F)>\epsilon(Q)$. It remains to show that if $g\in K[x]\setminus\{0\}$ has degree smaller than $\deg(F)$, then $\epsilon(g)\leq\epsilon(Q)$ for some $Q\in\Psi_\alpha$. Since $\deg(g)<\deg(F)$, there exists $Q\in\Psi_\alpha$ such that $\nu_Q(g)=\nu(g)$ and $\nu_Q(\partial_bg)=\nu(\partial_bg)$ for every $b\in\mathbb{N}$. Hence, by Corollary \ref{2vq>=v-be} \textbf{(ii)}, for every $b\in\mathbb{N}$ we have
\[
\nu(\partial_bg)=\nu_Q(\partial_bg)\geq\nu_Q(g)-b\epsilon(Q)=\nu(g)-b\epsilon(Q).
\]
This implies that $\epsilon(g)\leq\epsilon(Q)<\epsilon(F)$ and the result follows.
\end{proof}

\begin{Def}
A key polynomial $F\in K[x]$ for $\nu$ as in Proposition \ref{4naofixados} is a \textbf{successor key polynomial of $\Psi_\alpha$}. If such $F$ exists, then it follows from Proposition 3.1 of \cite{Nov} and the fact that $F$ is a key polynomial that $\nu(\Psi_\alpha)$ is bounded (by $\alpha\epsilon(F)$). If $\Psi_\alpha$ does not have maximum, then $F$ is a \textbf{limit key polynomial for $\Psi_\alpha$}. By Proposition \ref{4naofixados}, $F$ is a key polynomial for $\nu$ of degree larger than $\alpha$.
\end{Def}

Suppose that $F\in K[x]\setminus\{0\}$ is a successor of $\Psi_\alpha$ and set $B=\sup\nu(\Psi_\alpha)$ (such $\sup$ exists because $\nu(\Psi_\alpha)$ is bounded and $\nu$ is a rank one valuation).

\begin{Lema}\label{4removedordemonomios}
Assume that $f\in S_\alpha$ and take $h\in K[x]$ such that there exists $Q'\in\Psi_\alpha$ with $\nu_Q(f)<\nu_{Q'}(h)$, for every $Q\in\Psi_\alpha$. Then $f+h\in S_\alpha$.
\end{Lema}

\begin{proof}
For every $Q\in\Psi_\alpha$ such that $\nu(Q)\geq \nu(Q')$ we have $\nu_Q(f+h)=\nu_Q(f)<\nu(f)$. Also, by Lemma \ref{2sujestao} \textbf{(i)}
\[
\nu_Q(f)<\nu_{Q'}(h)\leq\nu_Q(h)\leq\nu(h).
\]
Hence, $\nu_Q(f+h)<\nu(f+h)$ which shows the result.
\end{proof}

\begin{Prop}\label{delta=deg}
For every $Q\in\Psi_\alpha$ we have $\delta_Q(F)=\deg_Q(F)$.
\end{Prop}

\begin{proof}
Let $Q_0\in\Psi_\alpha$ such that $\nu(Q_0)\geq 0$ and suppose, aiming for a contradiction, that $\delta:=\delta_{Q_0}(F)<\deg_{Q_0}(F)=:d$. Let $V_\delta$ and $V_d$ be the values of the coefficients of degree $\delta$ and $d$ in the $Q_0$-expansion of $F$, respectively. By Lemma \ref{2sujestao} \textbf{(iii)} we have
\[
\nu_Q(F)\leq V_\delta+\delta B<V_d+dB,
\]
where the second inequality holds because
\[
V_\delta+\delta\nu(Q_0)<V_d+d\nu(Q_0), \nu(Q_0)\geq 0\mbox{ and }\delta<d.
\]
Hence, there exists $Q_1\in\Psi_\alpha$ such that $\nu_Q(F)<V_d+d\nu(Q_1)$, for every $Q\in\Psi_\alpha$. Let $F=f_dQ_1^d+\ldots+f_0$ be the $Q_1$-expansion of $F$. By Lemma \ref{2sujestao} \textbf{(v)} we have $\nu(f_dQ_1^d)=V_d+d\nu(Q_1)$. Hence, by Lemma \ref{4removedordemonomios} we obtain that $F-f_dQ_1^d\in S_\alpha$. Since $\deg(F-f_dQ_1^d)<\deg(F)$, this contradicts the minimality of the degree of $F$ among polynomials in $S_\alpha$. Therefore, $\delta=d$ and the result follows from Lemma \ref{2sujestao} \textbf{(iv)}.
\end{proof}

\begin{Prop}\label{coeficientelider1}
For every $Q\in\Psi_\alpha$ the leading coefficient of the $Q$-expansion of $F$ is $1$.
\end{Prop}

\begin{proof}
If $\nu(\Psi_\alpha)$ has a maximum, then let $Q_1\in\Psi_\alpha$ be such that $\nu(Q_1)\geq\nu(Q)$ for every $Q\in\Psi_\alpha$. Otherwise, let $Q_0,Q_1\in\Psi_\alpha$ such that
\[
\nu(Q_1)-\nu(Q_0)>\deg_{Q_1}(F)(B-\nu(Q_1)).
\]

Let $F=f_\delta Q_1^\delta+\ldots+f_0$ be the $Q_1$-expansion of $F$. We will show that $f_\delta=1$. Indeed, if $Q\in\Psi_\alpha$ and $Q_1=Q+h$ is the $Q$-expansion of $Q_1$, replacing $Q_1=Q+h$ in the $Q_1$-expansion of $F$, for each $i$, $0\leq i\leq \delta-1$, the $Q$-expansion of $f_i(Q+h)^i$ does not have a term of degree $\delta$. On the other hand, in the $Q$-expansion of $(Q+h)^\delta$ the coefficient of degree $\delta$ is $1$. Therefore, the result follows.

Suppose, aiming for a contradiction, that $f_\delta\neq 1$. Since $F$ is monic this implies that $\deg(F)>\delta\alpha$. By Corollary \ref{irredutivel}, $Q_1$ is irreducible. Hence, by Bezout's identity, there exists a polynomial $a\in K[x]$, of degree less than $\alpha$, such that $af_\delta=1+q_\delta Q_1$. Here $q_\delta\in K[x]$ can be chosen of degree less than $\alpha$. Let 
\[
aF=(q_0Q_1+r_0)+\ldots+(q_{\delta-1}Q_1^\delta+r_{\delta-1}Q_1^{\delta-1})+(q_\delta Q_1^{\delta+1}+Q_1^\delta)
\]
be the $Q_1$-expansion of $aF$, where for each $i$, $0\leq i\leq \delta$, $af_i=q_iQ_1+r_i$ is the $Q_1$-expansion of $af_i$. By Lemma \ref{2sujestao} \textbf{(iii)} and Proposition \ref{delta=deg} we have
\[
\nu_Q(aF)=\nu(a)+\nu_Q(F)\leq\nu(af_\delta)+\delta B\mbox{ for every }Q\in\Psi_\alpha.
\]
We will show that $\nu(q_{\delta-1}Q_1^\delta)$ and $\nu(q_\delta Q_1^{\delta+1})$ are larger than $\nu(af_\delta)+\delta B=\delta B$ and the result will follow. Indeed, by Lemma \ref{4removedordemonomios} $aF-q_{\delta-1}Q_1^\delta-q_\delta Q_1^{\delta+1}\in S_\alpha$ and has degree $\delta\gamma<\deg(F)$. This is a contradiction to the minimality of the degree of $F$ in $S_\alpha$. We will divide the proof in the cases where $\nu(\Psi_\alpha)$ has and does not have a maximum.

If $\nu(Q_1)=B$, then by Lemma \ref{prod} and Proposition \ref{delta=deg} we have
\[
\nu(q_{\delta-1}Q_1^\delta)>\nu(r_{\delta-1}Q_1^{\delta-1})=\nu(af_{\delta-1}Q_1^{\delta-1})\geq\nu(af_\delta Q_1^\delta)=\delta B
\]
and
\[
\nu(q_\delta Q_1^{\delta+1})>\nu(Q_1^\delta)=\delta B.
\]
Hence, the result follows.

If $\nu(Q_1)-\nu(Q_0)>\delta(B-\nu(Q_1))$, then by Lemma \ref{2sujestao} \textbf{(vi)} we have
\[
\nu(q_\delta Q_1^{\delta+1})>\nu(af_\delta Q_1^\delta)+(\nu(Q_1)-\nu(Q_0))>\delta\nu(Q_1)+\delta(B-\nu(Q_1))=\delta B.
\]
Then, by Lemma \ref{2sujestao} \textbf{(vi)} and Proposition \ref{delta=deg} we have
\begin{displaymath}
\begin{array}{rcl}
\nu(q_{\delta-1}Q_1^\delta)&>&\nu(af_{\delta-1}Q_1^{\delta-1})+(\nu(Q_1)-\nu(Q_0))\geq\nu(af_\delta Q_1^\delta)+(\nu(Q_1)-\nu(Q_0))\\[8pt]
&>& \delta\nu(Q_1)+\delta(B-\nu(Q_1))=\delta B.
\end{array}
\end{displaymath}
Hence, the result follows.
\end{proof}

\section{The structure of the limit key polynomial}

For this section, let $\Psi_\alpha$ be the set of key polynomials for $\nu$ of degree $\alpha$. Assume that $\Psi_\alpha$ is not empty, does not have a maximum and that $F\in K[x]$ is a limit key polynomial for $\Psi_\alpha$. Then $\nu(\Psi_\alpha)$ is bounded and since $\nu$ is a rank one valuation we can set $B=\sup\nu(\Psi_\alpha)$ and $\epsilon(B)=\sup\epsilon(\Psi_\alpha)$.

By Proposition \ref{delta=deg}, for every $Q\in\Psi_\alpha$ we have $\delta:=\delta_Q(F)=\deg_Q(F)$. By Proposition \ref{coeficientelider1}, for every $Q\in\Psi_\alpha$ the leading coefficient in the $Q$-expansion of $F$ is $1$. Hence, for every $Q\in\Psi_\alpha$ we have $\nu_Q(F)=\delta\nu(Q)<\delta B$.

\begin{Not}
Let $Q\in\Psi_\alpha$ and $F=Q^\delta+\ldots+f_1Q+f_0$ be the $Q$-expansion of $F$. We set
\begin{enumerate}
\item $\gamma_Q(F)\in\{0,\ldots,\delta\}$ is the highest integer $i$, not a power of $p$, for which 
\[
\nu\left(f_iQ^i\right)<\delta B;
\] 
\item
\[
\Lambda_Q(F):=\{i\mid \gamma_Q(F)+1\leq i\leq \delta, i\mbox{ is a power of }p\mbox{ and }\nu(f_iQ^i)<\delta B\};
\]
\item if $\gamma_Q(F)\neq 0$, then
\[
\omega_Q(F):=\delta_Q\left(f_1Q+\ldots+f_{\gamma_Q(F)}Q^{\gamma_Q(F)}\right).
\]
\end{enumerate}

\end{Not}

\begin{Obs}
The value $\gamma_Q(F)$ is well-defined because $\nu_Q(f_0)<\delta B$. Indeed, if this were not the case, then by Lemma \ref{4removedordemonomios}, $F-f_0$ would be a limit key polynomial for $\Psi_\alpha$ and divisible by $Q$. This is a contradiction to Corollary \ref{irredutivel}.
\end{Obs} 

Let $Q_0\in\Psi_\alpha$ and $b_\infty\in\mathbb{N}$ as in Proposition \ref{bfinal}. Let $\Psi\subset\Psi_\alpha$ be the final subset of $\Psi_\alpha$ defined as the key polynomials $Q\in\Psi_\alpha$ for which
\[
\nu(p)>\delta(B-\nu(Q)), \epsilon(Q)-\epsilon(Q_0)>\delta(B-\nu(Q))\mbox{ and }I(Q)=\{b_\infty\}.
\]

\begin{Lema}\label{5canhao}
Take $Q_1,Q_2\in\Psi$ such that $\nu(Q_1)\leq\nu(Q_2)$, $F=Q_1^\delta+\ldots+a_0$ the $Q_1$-expansion of $F$ and $n\in\{1,\ldots,\delta\}$ a power of $p$. Let
\[
a_nQ_1^n=\sum_{i=0}^nb_iQ_2^i
\]
be the $Q_2$-expansion of $a_nQ_1^n$. For every $i$, $1\leq i\leq n-1$, we have 
\[
\nu\left(b_iQ_2^i\right)\geq\nu_{Q_1}(b_iQ_2^i)>\delta B.
\]
\end{Lema}

\begin{proof}
We have $\nu(a_nQ_1^n)\geq\nu_{Q_1}(F)=\delta\nu(Q_1)$ and hence $\nu(a_n)\geq(\delta-n)\nu(Q_1)$. Let $Q_1=Q_2+h$ be the $Q_2$-expansion of $Q_1$. Since $\nu(Q_2)\geq\nu(Q_1)$ we have $\nu(h)\geq\nu(Q_1)$. Substituting $Q_2+h$ in $a_nQ_1^n$ we have
\[
a_nQ_1^n=\sum_{i=0}^n\left({n\choose i}a_nh^{n-i}Q_2^i\right)=\sum_{i=0}^nb_iQ_2^i.
\]
In order to show the result, we will show that for every $i$, $0\leq i\leq n$, the terms of degree less than $n$ in the $Q_2$-expansions of ${n\choose i}a_nh^{n-i}Q_2^i$ have values, with respect to $\nu_{Q_1}$, greater than $\delta B$. This result is trivial for the term ${n\choose n}a_nQ_2^n=a_nQ_2^n$.

For $i=0$, let $c_0+\ldots+c_nQ_2^n$ be the $Q_2$-expansion of ${n\choose 0}a_nh^n$. We have
\[
\nu(c_0)=\nu\left({n\choose 0}a_nh^n\right)\geq\nu(a_n)+n\nu(Q_1)\geq\delta\nu(Q_1).
\]
Hence, by Lemma \ref{2sujestao} \textbf{(vi)} for every $j$, $1\leq j\leq n$, we have
\begin{displaymath}
\begin{array}{rcl}
\nu(c_jQ_2^j)&\geq&\nu_{Q_1}(c_jQ_2^j)>j(\nu(Q_1)-\nu(Q_2))+j(\nu(Q_2)-\nu(Q_0))+\nu(c_0)\\[8pt]
&\geq & \nu(Q_1)-\nu(Q_0)+\nu(c_0)>\delta(B-\nu(Q_1))+\delta\nu(Q_1)=\delta B,
\end{array}
\end{displaymath}
and the case $i=0$ is concluded.

For $i$, $1\leq i\leq n-1$, since $p\mid{n\choose i}$ and $\nu_{Q_1}(Q_2)=\nu(Q_1)$ we have
\[
\nu_{Q_1}\left({n\choose i}a_nh^{n-i}Q_2^i\right)\geq\nu(p)+\nu(a_nh^{n-i})+\nu_{Q_1}(Q_2^i)>\delta(B-\nu(Q_1))+\delta\nu(Q_1)=\delta B,
\]
and the result follows from Lema \ref{2sujestao} \textbf{(ii)}.
\end{proof}

\begin{Lema}\label{7essencial}
Take $Q_1,Q_2\in\Psi$ such that $\nu(Q_1)\leq\nu(Q_2)$ and $F=Q_1^\delta+\ldots+b_0$ and $F=Q_2^\delta+\ldots+c_0$ the $Q_1$ and $Q_2$-expansions of $F$, respectively. Then we have the following.
\begin{description}
\item[(i)] If $i$, $\gamma_{Q_1}(F)+1\leq i\leq \delta$, is such that $\nu(b_iQ_1^i)\geq\delta B$, then $\nu(c_iQ_2^i)\geq\delta B$. Moreover, if $\nu(Q_2)>\nu(Q_1)$, then $\nu(c_iQ_2^i)>\delta B$.
\item[(ii)] If $i$, $1\leq i\leq \delta$, is such that $\nu(b_iQ_1^i)<\delta B$ and $\nu(b_iQ_1^i)<\nu(b_jQ_1^j)$ for every $j$, $i+1\leq j\leq \delta$, which is not a power of $p$, then $\nu(b_i)=\nu(c_i)$.
\item[(iii)] If for every $Q_2\in\Psi$ such that $\nu(Q_2)\geq\nu(Q_1)$ we have $\nu(b_i)=\nu(c_i)$, $\nu(b_iQ_1^i)<\delta B$ and $\nu(c_iQ_2^i)<\delta B$, then $\nu(b_i)=\nu(c_i)=(\delta-i)B$.
\end{description}
\end{Lema}

\begin{proof}
Take $i$, $\gamma_{Q_1}(F)+1\leq i\leq \delta$, such that $\nu(b_iQ_1^i)\geq\delta B$. For each $j$, $0\leq j\leq \delta$, we will study the values of the terms $c_{i,j}Q_2^i$ of degree $i$ in the $Q_2$-expansion of $b_jQ_1^j$. This is enough because $c_iQ_2^i=(c_{i,0}+\ldots+c_{i,\delta})Q_2^i$. 

For $j<i$, the $Q_2$-expansion of $b_jQ_1^j$ has no term of degree $i$. Hence, $c_{i,j}=0$.

For $j=i$, by Lemma \ref{2sujestao} \textbf{(v)}, we have $\nu(c_{i,i})=\nu(b_i)$. Consequently, $\nu(c_{i,i} Q_2^i)\geq\nu(b_iQ_1^i)\geq\delta B$, with strict inequality if $\nu(Q_2)>\nu(Q_1)$.

For $j>i$ and $j$ a power of $p$, by Lemma \ref{5canhao} we have $\nu(c_{i,j}Q_2^i)>\delta B$.

For $j>i$ and $j$ not power of $p$, we have $\nu(b_jQ_1^j)\geq\delta B$. Hence, by Lemma \ref{2sujestao} \textbf{(ii)} we have
\[
\nu(c_{i,j}Q_2^i)\geq\nu_{Q_1}(c_{i,j}Q_2^i)\geq\nu(b_jQ_1^j)\geq\delta B.
\]
Moreover, the inequality is strict when $\nu(Q_2)>\nu(Q_1)$.

Since
\[
\nu(c_iQ_2^i)\geq\min_{0\leq j\leq\delta}\{\nu(c_{i,j}Q_2^i)\},
\]
we obtain \textbf{(i)}.

Take $i$, $1\leq i\leq \delta$, such that $i$ satisfies the assumptions of \textbf{(ii)}. For each $j$ $0\leq j\leq \delta$, we will study the values of the coefficients $c_{i,j}$ of degree $i$ in the $Q_2$-expansions of $b_jQ_1^j$. This is enough because $c_i=c_{i,0}+\ldots+c_{i,\delta}$. 

For $j<i$, the $Q_2$-expansion of $b_jQ_1^j$ has no term of degree $i$. Hence, $c_{i,j}=0$.

For $j=i$, by Lemma \ref{2sujestao} \textbf{(iii)} we have $\nu(c_{i,i})=\nu(b_i)$.

For $j>i$ such that $j$ is a power of $p$, by Lemma \ref{5canhao} we have
\[
\nu_{Q_1}(c_{i,j}Q_2^i)=\nu(c_{i,j}Q_1^i)>\delta B>\nu(b_iQ_1^i).
\]
Hence, $\nu(c_{i,j})>\nu(b_i)$.

For $j>i$ such that $j$ is not a power of $p$, by Lemma \ref{2sujestao} \textbf{(ii)} we have
\[
\nu_{Q_1}(c_{i,j}Q_2^i)=\nu(c_{i,j}Q_1^i)\geq\nu(b_jQ_1^j)>\nu(b_iQ_1^i).
\]
Hence, $\nu(c_{i,j})>\nu(b_i)$.

Therefore, we obtain $\nu(c_i)=\nu(c_{i,i})=\nu(b_i)$ and this shows \textbf{(ii)}.

We will show \textbf{(iii)} by contradiction. If $\nu(b_i)<(\delta-i)B$, then we can take $Q_2\in\Psi$ such that
\[
\nu(Q_2)\geq\nu(Q_1)\mbox{ and }(\delta-i)\nu(Q_2)>\nu(b_i)=\nu(c_i).
\]
This is a contradiction to $\nu(c_iQ_2^i)<\delta\nu(Q_2)=\nu_{Q_2}(F)$. If $\nu(b_i)>(\delta-i)B$, then we can take $Q_2\in\Psi$ such that
\[
\nu(Q_2)\geq\nu(Q_1)\mbox{ and }\nu(b_i)+i\nu(Q_2)>\delta B.
\]
This is a contradiction to $\nu(c_iQ_2^i)>\delta B$. Therefore, the result follows.
\end{proof}

\begin{Prop}\label{Lambda}
We have the following.
\begin{description}
\item[(i)] There exist $\gamma(F)\in\{0,\ldots,\delta\}$, $\Lambda(F)\subseteq\{\gamma(F)+1,\ldots,\delta\}$ and a final subset $\Psi_{\Lambda(F)}$ of $\Psi_\alpha$ such that for every $Q\in\Psi_{\Lambda(F)}$ we have
\[
\gamma(F)=\gamma_Q(F)\mbox{ and }\Lambda(F)=\Lambda_Q(F).
\]
\item[(ii)] For every $Q\in\Psi_{\Lambda(F)}$ and every $i\in\Lambda(F)$ the coefficient of degree $i$ in the $Q$-expansion of $F$ has value $(\delta-i)B$.
\item[(iii)] If $\gamma(F)\neq 0$, then for every $Q\in\Psi_{\Lambda(F)}$ the coefficent of degree $\gamma(F)$ in the $Q$-expansion of $F$ has value $(\delta-\gamma(F))B$.
\item[(iv)] If $\gamma(F)\neq 0$, then for every $Q\in\Psi_{\Lambda(F)}$ we have $\omega_Q(F)=\gamma(F)$.
\item[(v)] $\gamma(F)=0$.
\item[(vi)] If $Q\in\Psi_{\Lambda(F)}$ and $F=Q^\delta+\ldots+f_0$ is the $Q$-expansion of $F$, then
\[
P_{Q}(F):=\sum_{i\in\Lambda(F)}f_iQ^i+f_0
\]
is a limit key polynomial for $\Psi_\alpha$. Consequently, $\delta$ is a power of $p$.
\end{description}
\end{Prop}

\begin{proof}
By Lemma \ref{7essencial} \textbf{(i)}, for every $Q_1,Q_2\in\Psi$ such that $\nu(Q_1)\leq\nu(Q_2)$ we have $\gamma_{Q_2}(F)\leq\gamma_{Q_1}(F)$. Hence, there exists $\gamma(F)$, $0\leq \gamma(F)\leq \delta$, and a final subset $\Psi_{\gamma(F)}$ of $\Psi$ (consequently also final in $\Psi_\alpha$) such that for every $Q\in\Psi_{\gamma(F)}$ we have $\gamma(F)=\gamma_Q(F)$. Consequently, also by Lemma \ref{7essencial} \textbf{(i)}, for every $Q_1,Q_2\in\Psi_{\gamma(F)}$ such that $\nu(Q_1)\leq\nu(Q_2)$ we have
\[
\Lambda_{Q_2}(F)\subseteq\Lambda_{Q_1}(F).
\]
Consequently, there exists $\Lambda(F)\subseteq\{\gamma(F)+1,\ldots,\delta\}$ and a final subset $\Psi_{\Lambda(F)}$ of $\Psi_{\gamma(F)}$ (and consequently also final in $\Psi_\alpha$) such that for every $Q\in\Psi_{\Lambda(F)}$ we have $\Lambda(F)=\Lambda_Q(F)$. This shows \textbf{(i)}.

The items \textbf{(ii)} and \textbf{(iii)} are direct consequence of Corollary \ref{7essencial} \textbf{(ii)} and \textbf{(iii)}.

For \textbf{(iv)}, let $Q_1\in\Psi_{\Lambda(F)}$, $F=Q_1^\delta+\ldots+b_0$ the $Q_1$-expansion of $F$, $\omega=\omega_{Q_1}(F)$ and $\gamma=\gamma(F)$. Since 
\[
\nu(b_\gamma)=(\delta-\gamma)B\mbox{ and }\nu(b_\omega Q_1^\omega)\leq\nu(b_\gamma Q_1^\gamma)
\]
we have
\[
\nu(b_\omega)\leq(\delta-\gamma)B+(\gamma-\omega)\nu(Q_1)\leq(\delta-\omega)B.
\]
The second inequality holds beacuse $\omega\leq\gamma$  and the fact that if $\nu(b_\omega)=(\delta-\omega)B$, then $\omega=\gamma$. Since $\nu(b_\omega Q_1^\omega)<\delta B$, by Corollary \ref{7essencial} \textbf{(ii)} for every $Q_2\in\Psi_{\Lambda(F)}$ with $Q_2\geq Q_1$, if $c_\omega$ is the coefficient of degree $\omega$ in the $Q_2$-expansion of $F$, then
\[
\nu(c_\omega)=\nu(b_\omega)\leq(\delta-\omega)B.
\]
Consequently, $\nu(c_\omega Q_2^\omega)<\delta B$. Hence, by Corollary \ref{7essencial} \textbf{(iii)} we have $\nu(b_\omega)=(\delta-\omega)B$. Therefore, $\gamma=\omega$.

We will show \textbf{(v)} by contradiction. Suppose that $\gamma(F)\neq 0$ and write $\gamma(F)=p^eu$ where $e\in\mathbb{N}_0$, $u\in\mathbb{N}$ and $p\nmid u$. Since $\gamma(F)$ is not a power of $p$ we have $\gamma(F)-p^e>0$. Set $b=p^eb_\infty$. It is enough to show that, for every $Q\in\Psi_{\Lambda(F)}$ we have $\delta_Q(\partial_bF)>0$. Indeed, by Lemma \ref{2sujestao} \textbf{(i)} and \textbf{(vii)}, since $\Psi_\alpha$ does not have a maximum, this would imply that $\partial_bF\in S_\alpha$. This is a contradiction to the minimality of the degree of $F$ in $S_\alpha$.

Let $\Gamma=\{1,\ldots,\delta\}\setminus\Lambda(F)$, $Q\in\Psi_{\Lambda(F)}$, $F=Q^\delta+\ldots+f_0$ the $Q$-expansion of $F$ and
\[
F_1=\sum_{i\in\Gamma}f_iQ^i.
\]
By \textbf{(iv)} we have $\delta_Q(F_1)=\gamma(F)$. Consequently, by \textbf{(iii)} and Proposition \ref{2dp} we have
\[
\nu_Q(\partial_bF_1)=(\delta-\gamma(F))B+\gamma(F)\nu(Q)-b\epsilon(Q)\mbox{ and }\delta_Q(\partial_bF_1)=\delta-p^e.
\]
Since $\nu(f_0)\geq\nu_Q(F)=\delta\nu(Q)$, we have
\[
\nu_Q(\partial_bf_0)>\delta\nu(Q)-b\epsilon(Q_0)>\delta B-b\epsilon(Q)>\nu_Q(\partial_bF_1).
\]
Since
\[
\partial_bF=\partial_b(f_0)+\partial_b(F_1)+\sum_{i\in\Lambda(F)}\partial_b(f_iQ^i),
\]
it is enough to show that, for every $n\in\Lambda(f)$, the terms of degree $0$ and $\gamma(F)-p^e$ in the $Q$-expansion of $\partial_b(f_nQ^n)$ have values larger than $\nu_Q(\partial_bF_1)$.

Take $n\in\Lambda(F)$. By (\ref{2derivada}), it is enough to show that for every $\lambda=(b_0,\ldots,b_r)\in\mathcal{S}_{b,n}$ for which $n-r\leq\gamma(F)-p^e$ we have
\[
\nu(C(\lambda)T_\lambda(f_nQ^n))>\nu_Q(\partial_bF_1).
\]
By \textbf{(ii)} we have $\nu(f_n)=(\delta-n)B$. If $n-r<\gamma(F)-p^e$, the by Lemma \ref{Lemafinal} we have
\[
\nu(T_\lambda(f_nQ^n))\geq(\delta+1-\gamma(F))B+(\gamma(F)-1)\nu(Q)-b\epsilon(Q)>\nu_Q(\partial_bF_1).
\]
If $n-r=\gamma(F)-p^e$, then since $0<r<n$ and $n$ is a power of $p$, we have $p\mid{n\choose r}$. Consequently, $p\mid C(\lambda)$. Therefore, 
\begin{displaymath}
\begin{array}{rcl}
\nu(C(\lambda)T_\lambda(f_nQ^n))&\geq&\nu(p)+(\delta-n)B+n\nu(Q)-b\epsilon(Q)\\[8pt]
&>&\delta(B-\nu(Q))+(\delta-n)B+n\nu(Q)-b\epsilon(Q)\geq\delta B-b\epsilon(Q)\\[8pt]
&\geq&\nu_Q(\partial_bF_1).
\end{array}
\end{displaymath}
The result follows from this.

By \textbf{(v)} and Lemma \ref{4removedordemonomios} we have that $P_Q(F)$ is a key polynomial for $\Psi_\alpha$. Hence, $\deg(P_Q(F))=\deg(F)$. Consequently, $\delta\in\Lambda(F)$ and hence $\delta$ is a power of $p$. This shows \textbf{(vi)}.
\end{proof}

\ \\
\noindent{\footnotesize JOSNEI NOVACOSKI\\
Departamento de Matem\'atica--UFSCar\\
Rodovia Washington Lu\'is, 235\\
13565-905 - S\~ao Carlos - SP\\
Email: {\tt josnei@dm.ufscar.br} \\\\

\noindent{\footnotesize MICHAEL DE MORAES\\
Departamento de Matem\'atica--ICMC-USP\\
Av. Trabalhador s\~ao-carlense, 400\\
13566-590 - S\~ao Carlos - SP\\
Email: {\tt michael.moraes@usp.br} \\\\
\end{document}